\documentclass[12pt]{amsart}
\usepackage[utf8]{inputenc}

\usepackage{amsmath, amsthm, amscd, amssymb, graphicx}
\usepackage{tikz}
\usepackage{scalefnt}
\newtheorem{thm}{Theorem}[section]
\newtheorem*{thm*}{Theorem}

\newtheorem*{cor*}{Corollary}
\newtheorem{lem}[thm]{Lemma}
\newtheorem{prop}[thm]{Proposition}
\newtheorem*{prop*}{Proposition}

\newtheorem*{con*}{Conjecture}

\theoremstyle{definition}
\newtheorem{defn}[thm]{Definition}
\newtheorem*{defn*}{Definition}

\theoremstyle{remark}

\newtheorem*{example*}{Example}

\newcommand{\R}{\mathbb{R}}
\newcommand{\Z}{\mathbb{Z}}
\newcommand{\link}{\mbox{lk}}
\newcommand{\prob}{\mathbb P}
\newcommand{\expect}{\mathbb E}
\newcommand{\eps}{\epsilon}

\begin{document}
\bibliographystyle{plain}

\title{$2$-complexes with large $2$-girth}

\bigskip

\author[Dotterrer]{Dominic Dotterrer}
\author[Guth]{Larry Guth}
\thanks{L.G.\  was supported in part by a Simons Investigator Grant.}
\author[Kahle]{Matthew Kahle}
\thanks{M.K.\ gratefully acknowledges support from DARPA \#N66001-12-1-4226, NSF \#CCF-1017182 and \#DMS-1352386, the Institute for Mathematics and its Applications, and the Alfred P.\ Sloan Foundation.}
\thanks{All three authors thank the Institute for Advanced Study for hosting them in Spring 2011, when some of this work was completed.}


\maketitle

\begin{abstract} The 2-girth of a 2-dimensional simplicial complex $X$ is the minimum size of a non-zero 2-cycle in $H_2(X, \Z/2)$.  We consider the maximum possible girth of a complex with $n$ vertices and $m$ 2-faces.  If $m = n^{2 + \alpha}$ for $\alpha < 1/2$, then we show that the 2-girth is at most $4 n^{2 - 2 \alpha}$ and we prove the existence of complexes with 2-girth at least $c_{\alpha, \eps} n^{2 - 2 \alpha - \eps}$.  On the other hand, if $\alpha > 1/2$, the 2-girth is at most $C_{\alpha}$.  So there is a phase transition as $\alpha$ passes 1/2.

Our results depend on a new upper bound for the number of combinatorial types of triangulated surfaces with $v$ vertices and $f$ faces.

\end{abstract}

\section{Introduction} 

In this paper, we study a 2-dimensional analogue of the girth of a graph.  Recall that the girth of a graph is the shortest length of a non-trivial cycle in the graph.  If $X$ is a 2-dimensional simplicial complex, then the {\it $2$-girth} of $X$ is defined to be the minimum number of 2-faces in a non-zero 2-cycle with coefficients in $\Z/2$.  We address the following question: if $X$ has $n$ vertices and $m$ 2-faces, what is the maximum possible size of the 2-girth of $X$?  We prove upper and lower bounds which match pretty closely.  They show a phase transition as $m$ passes $n^{2.5}$.

For context, let us first recall the situation for graphs.  Let $g(n,m)$ denote the maximum possible girth of a graph with $n$ vertices and $m$ edges.  We note that $g(n,n) = n$, because any graph with $n$ vertices and $n$ edges must contain a cycle, and the number of edges in the cycle is at most the total number of edges in the graph.  As we increase the number of edges, $g(n,m)$ decreases rapidly.  For instance, $g(n, 2n) = O(\log n)$, and for any $\alpha > 0$, $g(n, n^{1+\alpha}) \le C_{\alpha}$.  These estimates are special cases of the Moore bounds (cf.  \cite{alon2002moore} ).  

Now we turn to 2-dimensional simplicial complexes.  Let $g_2(n,m)$ denote the maximum possible 2-girth of a 2-dimensional simplicial complex $X$ with $n$ vertices and $m$ 2-faces.  A graph with $n$ vertices has at most $n \choose 2$ edges.  If $m = {n \choose 2} - n + 2$, then a dimension counting argument implies that $H_2(X, \Z/2) \not= 0$, and so $X$ contains a non-zero 2-cycle.  Therefore,

$$ g_2(n, {n \choose 2} - n+ 2) \le {n \choose 2} - n + 2 \le n^2. $$

\noindent As we add 2-faces to a simplicial complex $X$, its 2-girth can only decrease, and so $g_2(n,m)$ is decreasing in $m$.  In analogy with the case of graphs, one might expect $g_2(n, m)$ to decrease very sharply as $m$ increases.  For instance, we initially expected that $g_2(n, 2n^2) = O(\log n)$.  But it turns out that this is not the case.  Aronshtam, Linial, Luczak, and Meshulam introduced techniques in \cite{ALLM13} which show that $g_2(n, Kn^2) \ge c_K n^2$ for every constant $K >0$ -- see Section \ref{sec:comm} for more discussion of their work.

We study the behavior of $g_2(n,m)$ as $m$ increases further.   We prove two upper bounds for $g_2(n,m)$.  

\vskip10pt

{\bf \noindent Theorem \ref{thm:syst}.}
If $\alpha \ge 0$ and
$$m \ge n^{2 + \alpha},$$
then
$$g_2(n,m) \le 4 n^{2 - 2 \alpha}.$$\\

In the case $\alpha > 1/2$, there is a much better upper bound on $g_2(n,m)$. This was studied implicitly by S\'os, Erd\H{o}s, and Brown \cite{SEB73}, in the context of Tur\'an theory for hypergraphs. We include their argument for the sake of completeness, to show the following.\\

\vskip10pt

{\bf \noindent Theorem \ref{thm:syst2}.}
If $m \ge n^{2 + \alpha}$,
where $\alpha > 1/2$,
then
$$g_2(n, m) \le C_{\alpha},$$
where $C_{\alpha}$ is a constant which depends only on $\alpha$.\\

\bigskip

Our most difficult result is a lower bound on $g_2(n,m)$ that roughly matches the upper bound from Theorem \ref{thm:syst} in the regime $0 \le \alpha < 1/2$.  

\vskip10pt

{\bf \noindent Theorem \ref{thm:exist}.}
\label{thm:exist}
Let $0 < \alpha < 1/2$, and $\epsilon > 0$. For sufficiently large $n$, there exist simplicial complexes $\Delta$ with $n$ vertices and with at least $m = n^{2+\alpha}$ faces, and $2$-girth at least $n^{2-2\alpha-\epsilon}$.\\

\medskip

Notice that if $\alpha$ is slightly less than $1/2$, then $g_2(n, n^{2 + \alpha})$ is roughly $n$, while if $\alpha$ is slightly more than $1/2$, then $g_2(n, n^{2 + \alpha})$ is bounded by a constant.  Figure 1 compares the behavior of $g(n,m)$ and $g_2(n,m)$.  

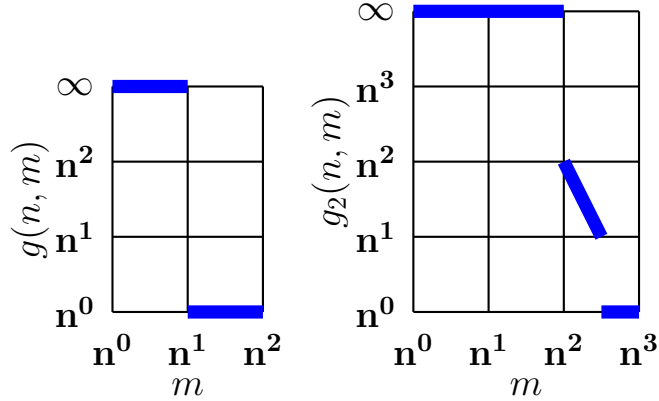
\begin{figure}
\centering
\begin{tikzpicture}[thick]
\large
\draw (1,1) -- (3,1);
\draw (1,2) -- (3,2);
\draw (1,3) -- (3,3);
\draw(1,4) -- (3,4);

\draw (1,1) -- (1,4);
\draw (2,1) -- (2,4);
\draw (3,1) -- (3,4);

\node at (1, 1/2) {$\bf n^0$};
\node at (2,1/2) {$\bf n^1$};
\node at (3,1/2) {$\bf n^2$};

\node at (1/2,1) {$\bf n^0$};
\node at (1/2,2) {$\bf n^1$};
\node at (1/2,3) {$\bf n^2$};
\node at (1/2,4) {$\bf \infty$};

\node at (2,0)  {$m$};
\node at (-0.1,2.5) {\rotatebox{90}{$g(n,m)$}};

\draw[blue, line width=5] (1,4) -- (2,4);
\draw[blue, line width=5] (2,1) -- (3,1);

\large
\draw (5,1) -- (8,1);
\draw (5,2) -- (8,2);
\draw (5,3) -- (8,3);
\draw(5,4) -- (8,4);
\draw(5,5) -- (8,5);

\draw (5,1) -- (5,5);
\draw (6,1) -- (6,5);
\draw (7,1) -- (7,5);
\draw (8,1) -- (8,5);

\node at (5, 1/2) { $\bf n^0$};
\node at (6,1/2) { $\bf n^1$};
\node at (7,1/2) { $\bf n^2$};
\node at (8,1/2) { $\bf n^3$};

\node at (9/2,1) {$\bf n^0$};
\node at (9/2,2) {$\bf n^1$};
\node at (9/2,3) { $\bf n^2$};
\node at (9/2,4) {$\bf n^3$};
\node at (9/2,5) {$\bf \infty$};

\node at (6.5,0) {$m$};
\node at (3.9,3) {\rotatebox{90}{$g_2(n,m)$}};

\draw[blue, line width=5] (5,5) -- (7,5);
\draw[blue, line width=5] (7,3) -- (7.5, 2);
\draw[blue, line width=5] (7.5,1) -- (8,1);

\end{tikzpicture}

\caption{Comparing the large-scale behavior of $g(n,m)$ and $g_2(n,m)$, on a log-log scale. }
\label{fig:asy}
\end{figure}

The proof of Theorem \ref{thm:exist} is based on random simplicial complexes.  Linial and Meshulam \cite{LM} introduced the following random model of simplicial complexes, called $Y(n,p)$.  Take the complete graph on $n$ vertices.  There are $n \choose 3$ possible 2-simplices that we could add to this graph.  Include each independently with probability $p$.  If $p = n^{\alpha - 1}$, then with high probability the resulting simplicial complex $Y$ has on the order of $n^{2 + \alpha}$ 2-faces.  We count inclusion-minimal 2-cycles in $Y$, and we show that with high probability $Y$ contains only a small number of inclusion-minimal 2-cycles of size less than $n^{2 - 2 \alpha - \eps}$.  By removing a 2-face from each of these cycles, we get a complex with 2-girth at least $n^{2 - 2 \alpha - \eps}$ and with roughly $n^{2 + \alpha}$ 2-faces.

Let $\Delta_n^{(2)}$ denote the 2-skeleton of the simplex on $n$ vertices.  We have $Y \subset \Delta_n^{(2)}$.  If $z$ is an inclusion-minimal 2-cycle in $\Delta_n^{(2)}$ with $f$ faces, then $z$ will be included in $Y \in Y(n,p)$ with probability $p^f$.  So to prove our main theorem, we have to count the number of inclusion minimal 2-cycles in $\Delta_n^{(2)}$.  This turns out to be a complicated counting problem.  We organize the inclusion minimal 2-cycles according to combinatorial type, which is defined as follows.  Given $z$ an inclusion minimal 2-cycle, there is a triangulated surface $\Sigma$ and a simplicial map $\Sigma \rightarrow \Delta_n^{(2)}$ realizing $z$.  Because $z$ is inclusion-minimal, $\Sigma$ is connected.  The combinatorial type of $z$ is the combinatorial type of the surface $\Sigma$.  It's not so hard to count the number of inclusion minimal 2-cycles of a fixed combinatorial type.  It turns out to be harder to count and organize the possible combinatorial types.  The most difficult step in the argument is an estimate about the set of possible combinatorial types.

\vskip10pt

{\bf \noindent Theorem \ref{thm:exist}.} Let $T(f,w)$ be the set of connected, triangulated closed surfaces with $f$ 2-faces and $w$ vertices, up to simplicial isomorphism.  For every $\delta > 0$, there is a constant $C_\delta$ so that

$$|T(f,w)| \le C_\delta^f f^{f/2} w^{-(1-\delta) w}. $$





%

Using this counting argument, we study the sizes of cycles in a random simplicial complex $Y$ chosen from $Y(n,p)$.  A similar method gives an estimate about the isoperimetric properties of $Y$:

\vskip10pt

{\bf \noindent Theorem \ref{thm:fill}}
Let $\alpha \in \left[ 0 , \frac{1}{2} \right)$ and $\epsilon > 0$. Suppose $p = n^{\alpha-1}$, so the expected number of faces in $Y(n,p)$ is $n^{2 + \alpha}$. Let $A$ be the filling area of the cycle $123$ in $Y$. Then with high probability
$$n^{2 - 2 \alpha - \epsilon} \le A \le n^{2 - 2 \alpha + \epsilon}.$$

\section{Notation and Concepts} \label{sec:notation}

\begin{defn}
Let $S = \{ 1, 2, \dots, n\}$. A {\it simplicial complex}, $\Delta$, on the underlying set $S$ is a collection of subsets $X \subseteq 2^S$ such that if $A \subseteq S$ is an element of $\Delta$ and $B \subset A$, then $B$ is an element of $\Delta$.\\

We refer to the elements of $\Delta$ of order $1$, $2$, and $2$ as {\it vertices}, {\it edges} and {\it faces}, respectively. We denote the set of vertices, edges and faces as $\Delta^{(0)}$, $\Delta^{(1)}$ and $\Delta^{(2)}$ respectively.\\

\end{defn}

\begin{defn}
If we set $$C_k \Delta : = \{ c: \Delta^{(k)} \to \mathbb{F} \},$$ we obtain a chain complex $$\cdots \to C_{k+1} \Delta \overset{\partial}{\to} C_k \Delta \overset{\partial}{\to} C_{k-1} \Delta \to \cdots$$

whose chain maps, $\partial = \partial_k : C_k \Delta \to C_{k-1} \Delta$ are defined by $$\partial c (\sigma) := \sum_{x \in S, \, \sigma \cup \{ x \} \in \Delta} c (\sigma \cap \{ x \}).$$ \\

We define $Z_k \Delta := \ker \partial_k$, the subspace of $k$-{\it cycles}, $B_{k-1} \Delta := {\rm im} \partial_k$, the subspace of $(k-1)$-{\it boundaries}, and $H_k (\Delta; \mathbb{F}) := Z_k \Delta / B_k \Delta,$ the $k$-th {\it homology} of $\Delta$.\\
\end{defn}

\begin{defn}
For $\mathbb{F} = \Z / 2 \Z$, we use the Hamming ``norm'': for $c \in C_k \Delta $, define $$\lVert c \rVert : = {\rm supp} (c) = \vert \{ \sigma \in \Delta^{(k)} \; \vert \; c (\sigma) = 1 \} \vert.$$

\end{defn}

\section{Upper bounds on $g_2(n,m)$} \label{sec:ineq}

\begin{thm}\label{thm:syst}
If $$m \ge n^{2 + \alpha},$$
where $0 < \alpha \le 1/2$, then
$$g_2(n,m) \le 4 n^{2 - 2 \alpha}.$$
\end{thm}

So Theorem \ref{thm:syst} bounds the size of the 2-girth for $m$ up to $n^{5/2}$. For $m$ much larger than this, the 2-girth is bounded in size.

\begin{thm} \label{thm:syst2} [S\'os, Erd\H{o}s, and Brown] Let $\Delta$ be a $2$-dimensional simplicial complex with $n$ vertices and $m$ faces. If $m \ge n^{2 + \alpha}$,
where $1/2 < \alpha <1$,
then the 2-girth of $\Delta$ is bounded by $C_\alpha$ where $C_{\alpha}$ is a constant which depends only on $\alpha$.
\end{thm}

Theorem \ref{thm:syst2} is a result of S\'os, Erd\H{o}s, and Brown \cite{SEB73}, but we include a proof here for the sake of completeness.

\subsection{Proof of Theorem \ref{thm:syst}}

\begin{lem} \label{lem:dim}
Let $\Delta$ be a $2$-dimensional complex with $n$ vertices and $m$ faces. If $m \geq n^2/2$, then $H_2 (\Delta) \neq 0$.
\end{lem}

\begin{proof}
Since $\Delta$ is a simplicial complex, the number of edges of $\Delta$ is at most ${n \choose 2} < n^2 / 2$. If $m \geq \frac{n^2}{2}$, then the boundary map in simplicial homology from $2$-chains to $1$-chains has a nontrivial cycle, in which case $H_2 (\Delta) \neq 0$.
\end{proof}

We will restrict our attention to $\Z / 2 \Z$-cycles, and we will restrict ourselves to the ones which are irreducible in the following sense:

\begin{defn}
A $\Z / 2 \Z$-cycle, $z \in Z_2 (\Delta)$ is called {\it inclusion-minimal} if for any other cycle $z' \in Z_2 ( \Delta)$, ${\rm supp} (z') \subseteq {\rm supp} (z)$ implies that $z' = z$.
\end{defn}

\begin{lem}\label{lem:small}
For a $2$-dimensional complex on $n$ vertices, every inclusion-minimal $2$-cycle has fewer than $n^2/2$ faces.
\end{lem}

\begin{proof}
If there are more than $n^2 / 2$ faces in a cycle, then delete faces arbitrarily until no more than $n^2 / 2$ faces remain. By Lemma \ref{lem:dim}, there still remains at least one cycle. So an inclusion-minimal cycle can not have more than $n^2 / 2$ faces.
\end{proof}

%
%
%
%

\begin{proof}[Proof of Theorem \ref{thm:syst}]

Let $D$ be a random subcomplex of $\Delta$, chosen as follows. Let $V$ be a set of exactly $k$ vertices, chosen uniformly over all $n \choose k$ such subsets, and let $D= D(V)$ be the induced subcomplex of $\Delta$ on $V$.

So $D$ is a simplicial complex with $k$ vertices. What is the expected number of $2$-faces $\expect [ f_2(D)]$? For a $2$-face $T = \{ x, y , z \}$ in $\Delta$, $T$ is a face in $D$ if and only if $x, y, z \in V$. We have
\begin{align*}
\prob [ x, y , z \in V ] &= \frac{{n -3 \choose k-3}}{{n \choose k}},\\
& \ge \frac{k^3}{2n^3},
\end{align*}
as long as $n \ge k \ge 6$.

Now,
\begin{align*}
\expect [ f_2(D)] & \ge \frac{k^3}{2n^3} f_2(\Delta)\\
& \ge \frac{k^3}{2n^3} n^{2+ \alpha}\\
\end{align*}

Setting $k =2 n^{1 - \alpha}$,
\begin{align*}
\expect [ f_2(D)] & \ge \frac{8 n^{3 - 3 \alpha}}{2n^3} n^{2+ \alpha}\\
& = 4 n^{2 - 2 \alpha}\\
& = k^2.
\end{align*}

Since this is the average, there must exist some subcomplex $D'$ on  $k$ vertices, such that $f_2(D') \ge k^2$.
Applying Lemma \ref{lem:dim}, we have that $H_2(D') \neq 0$. Applying Lemma \ref{lem:small}, $D'$ contains a non-trivial 2-cycle with at most $k^2$ faces.  Since $k = 2 n^{1 - \alpha}$, we get the desired bound.
\end{proof}

\subsection{Proof of Theorem \ref{thm:syst2}}

\begin{lem} \label{lem:moore}
If $H$ is a graph on fewer than $n$ vertices, and with at least $n^{1 + \beta}$ edges, with $\beta > 0$ then $H$ contains a cycle of length at most $$C_{\beta}= 2 \left\lceil \frac{1}{\beta} \right\rceil + 1.$$
\end{lem}

\begin{proof}[Proof of Lemma \ref{lem:moore}]
A graph $H$ of average degree $d$ contains a subgraph $H'$ of minimum degree $d' \ge d/2$. (See for example Proposition 1.2.2 in \cite{Diestel10}.) Here $$d= 2n^{1+\beta} / n = 2n^{\beta}$$ and we restrict our attention to a subgraph $H' \subseteq H$, which has minimum degree at least $n^{\beta}$.

Let $v \in H'$ be a vertex, and for $k=0,1,2, \dots$ let $N_v(k)$ denote the number of vertices in $H'$ at graph distance $k$ from $v$. By the minimum degree requirement, $N_v(1) \ge n^{\beta}$.  If any of the distance-$1$ vertices have an edge between them, we have $g(H) \le g(H') \le 3$. 

 If any two of the distance-$1$ vertices have a common neighbor, $g(H) \le 4$. Similarly, if any two of the distance-$2$ vertices are adjacent, then $g(H) \le 5$, and so on.  Now, every vertex at distance $d$ must have at least $n^{beta}-1$ neighbors at distance $d+1$.
 So up to distance $d$, we have at least
 $$ 1 + n^{\beta} +  n^{\beta} \left(  n^{\beta}  -  1  \right) +  n^{\beta} \left(  n^{\beta}  -  1  \right)^2 + \dots +  n^{\beta} \left(  n^{\beta}  -  1  \right)^{d-1}$$
\end{proof}

\begin{proof}[Proof of Theorem \ref{thm:syst2}]
Define the degree $\deg(e)$ of an edge $e$ to be number of $2$-dimensional faces containing it. Then
$$\sum_e \deg(e) = 3 m$$
and $m \ge n^{5/2 + \alpha}$
by assumption. Note that the average edge degree is then at least $\lfloor 6 n^{1/2 + \alpha} \rfloor$.

Let $P$ denote the number of pairs of $2$-dimensional faces in $\Delta$ which share an edge. Then
$$P = \sum_e {\deg(e) \choose 2}.$$

Note that if $a+2 \le b$, then
$${a+1 \choose 2} + {b-1 \choose 2} \le { a \choose 2} + {b \choose 2},$$
so the sum $\sum_e {\deg(e) \choose 2}$ is minimized when the parts are as close together as possible.

Then even if each edge $e$ had degree
$\lfloor 6 n^{1/2 + \alpha} \rfloor$, we still have
\begin{align*}
P &\ge {n \choose 2}{\lfloor 6 n^{1/2 + \alpha} \rfloor \choose 2}\\
& \ge n^{3 + 2 \alpha}
\end{align*}
for large enough $n$.

Now for every pair of vertices $a,b$, let $T(a,b)$ denote the graph $\link(a) \cap \link(b)$. I.e.\ the vertices in $T(a,b)$ correspond to the intersection of the neighborhoods of $a$ and $b$, and the edges $xy$ in $T(a,b)$ correspond to pairs of triangles $axy$, $bxy$.

If we sum up the number of edges in $T(a,b)$ over all ${ n \choose 2}$ pairs $a,b$, this gives $P$. So by the pigeonhole principle, there must be at least one pair $a,b$ such that $T(a,b)$ has at least $n^{1 + 2 \alpha}$ edges. Then applying Lemma \ref{lem:moore}, this graph must contain a cycle $C$ on at most $C_{alpha}$ edges.

In the simplicial complex $\Delta$ then, there is a suspension over $C$ with suspension points $a$ and $b$. So in particular there exists a $2$-cycle with at most $2 C_{\alpha}$ $2$-faces.

\end{proof}

\section{Complexes with large girth}

\subsection{Overview} \label{sec:largegirth}

We prove the existence of simplicial complexes with large girth.

\begin{thm} \label{thm:exist}
Let $0 < \alpha < 1/2$, and $\epsilon > 0$. For sufficiently large $n$, there exist simplicial complexes $\Delta$ with $n$ vertices and with at least $m=n^{2 + \alpha}$ faces, and $2$-girth at least $n^{2 - 2 \alpha - \epsilon}$.
\end{thm}

Our strategy is to start with a random $2$-dimensional simplicial complex.  We recall the Linial-Meshulam model of random simplicial complexes $Y(n,p)$ from \cite{LM}.  A complex $Y$ in $Y(n,p)$ has $n$ vertices, and $n \choose 2$ edges, and each $2$-face is included independently with probability $p$.  We choose $p = n^{-1 + \alpha}$, where $0 < \alpha < 1/2$ is as in Theorem \ref{thm:exist}.

We say that an event occurs {\it with high probability (w.h.p.)} if the probability approaches one as $n \to \infty$.  We will show that with high probability, a complex $Y$ in $Y(n,p)$ contains only a few small 2-cycles.  We construct the complex $\Delta$ in Theorem \ref{thm:exist} by starting with a random 2-dimensional complex $Y$ from $Y(n,p)$, and then removing a small number of 2-faces to destroy all the small 2-dimensional cycles.  

Let $\Delta_n^{(2)}$ denote the 2-skeleton of the simplex on $n$ vertices so that $Y \subset \Delta_n^{(2)}$.  If $z$ is a 2-cycle in $\Delta_n^{(2)}$ with $f$ 2-faces, then $z$ is included in $Y$ with probability $p^f$.  To estimate the number of small 2-cycles in $Y$, we have to organize the set of 2-cycles in $\Delta_n^{(2)}$.  It turns out to be a good idea to study inclusion-minimal 2-cycles.  We say that a 2-cycle $z$ is inclusion minimal if there is no 2-cycle supported on a proper subset of the 2-faces in the support of $z$.  

Let $Z(f,v)$ denote the number of inclusion-minimal $2$-cycles in $\Delta_n^{(2)}$ with vertex support $[v] =\{ 1, 2, \dots, v \}$, and with exactly $f$ faces.  

\begin{lem} \label{lem:range} $Z(f,v) = 0$ unless
$$ \sqrt{2f} \le v \le f/2 + 2.$$
\end{lem}

\begin{proof}
The Euler formula for a finite $2$-dimensional complex is that
$$v - e + f = \beta_0 - \beta_1 + \beta_2.$$
For a minimal $2$-cycle, we have that $\beta_0 = \beta_2 = 1$, so
$$\hspace{-0.4in} (*) \hspace{0.4in} v-e+f = 2 - \beta_1.$$
Since $e \le {v \choose 2}$, this gives that
$$f \le {v \choose 2} - v + 2 - \beta_1 \le \frac{v^2}{2},$$
so $v \ge \sqrt{2f}$.

For the other inequality, we first assume without loss of generality that every edge is contained in at least one $2$-dimensional face---if not, then the edge can be deleted without affecting either $v$ or $f$. Since it is a $2$-cycle, every edge is then contained in at least two $2$-faces, so $$2e \le 3f.$$

Combining with (*) to eliminate the variable $e$, we have that $$f \ge 2v - 4 + 2\beta_1.$$ Since $\beta_1 \ge 0$, the inequality $v \le f/2 + 2$ follows.\\
\end{proof}

We will also require an upper bound on $Z(f,v)$, namely that for any fixed $\delta > 0$, and large enough $f$ and $v$, we have that
$$Z(f,v) \le C_{\delta}^f f^{f/2} v^{\delta f} .$$
We establish this bound in Section \ref{sec:cyc}.

We will use $\alpha$ and $\epsilon$ consistently throughout the section. Since $\alpha < 1/2$, we may assume without loss of generality that $\epsilon$ is small enough that $2 \alpha + \epsilon < 1$. 

Now we define constants $\delta$ and $M$ which only depend on $\alpha$ and $\epsilon$, and which we also use throughout the rest of the section.

First we set $$\delta = \min \left\{ \frac{1 - 2 \alpha}{4}, \frac{\epsilon}{3 ( 2 - 2 \alpha - \epsilon)} \right\}.$$
and $$M = \left\lceil \frac{8}{1- 2 \alpha} \right\rceil + 1.$$

%
%



\bigskip

First, we consider \emph{small cycles}, where $4 \le v \le M.$ We will show that with high probability the total number of such cycles is small relative to the number of $2$-dimensional faces. Then we can delete one face out of every small cycle without significantly affecting the number of faces.

Next, we consider \emph{intermediate cycles}, where $2M-2 \le f \le n$. By Lemma \ref{lem:range}, if $v \ge M+1$ then $f \ge 2M-2$, so there is no gap between small and intermediate cycles. We show that with high probability there are no intermediate cycles.

Finally, we consider \emph{large cycles}, where $n \le f \le n^{2- 2 \alpha - \epsilon}.$ Again, we show that with high probability there are no such cycles.

The conclusion is that w.h.p., the modified random $2$-complex only has \emph{very-large} cycles, i.e.\ cycles of area greater than $n^{2- 2 \alpha - \epsilon}$. \\

\subsection{Small cycles} \label{sec:smallcycles}  For a simplicial complex $\Delta$ and number $M$, let $C_{\Delta}(M)$ denote the number of cycles supported on at most $M$ vertices.

Define a subcomplex of $Y$ on $v$ vertices as {\it barely-dense} if it has exactly $2v-4$ faces.
Let $T_{\Delta}(M)$ denote the number of barely-dense subcomplexes of $\Delta$ with at most $M$ vertices. Note that $T_{\Delta}(M) \ge C_{\Delta}(M)$ for every $\Delta$ and $M$. Indeed, every cycle has $f \geq 2v - 4$ and so contains a barely-dense subcomplex so we one can remove all the cycles by removing one face from every barely-dense subcomplex.

Again, letting $\Delta = Y = Y(n,p)$ and by linearity of expectation,

$$\expect \left[T_{Y}(M) \right] = \sum_{v = 4}^{M} {n \choose v} {{v \choose 3} \choose 2v-4} p^{2v-4}.$$

Since $M$ is a constant which only depends on  $\alpha$, this is a finite sum and $v$ is bounded.
Moreover, the terms in the sum are strictly decreasing as $v$ increases for large enough $n$, so the first term in the sum dominates.

Then
\begin{align*}
\expect \left[T_{Y}(M) \right] & \le n^4 p^{4}\\
& = n^{4} n^{-4 + 4 \alpha}\\
& = n^{4 \alpha}.\\
\end{align*}

On the other hand, the expected number of $2$-faces $\expect[ f_2(Y)]$ is
$${n \choose 3} p \approx \frac{n^{2 + \alpha}}{6}.$$
The number of faces is a binomial random variable, therefore it is tightly concentrated around its mean (by Chernoff bounds, for example).

By Markov's inequality,
$$\prob \left[ T_{Y}(M) \ge n^{2 + \alpha/2} \right] \le \frac{n^{4 \alpha}}{ n^{2 + \alpha/2} },$$
which tends to zero as $n \to \infty$ since $\alpha < 1/2$.

Since $T_Y(M)$ dominates $C_Y(M)$, with high probability we can remove one face from every small cycle and still be left with almost all of the faces.


%

\bigskip

\subsection{Intermediate cycles} The sum
$$\sum_{f=2M-2}^{n-4} \sum_{v = \sqrt{2f}}^{f/2 + 2} {n \choose v} Z(f,v) p^f$$
is a union bound on the probability that there are any intermediate cycles: cycles whose number of faces $f$ satisfies $2M-2 \le f \le n$.

Since $f \le n$ and $v \le f/2 + 2 \le n/2$, we have that
$${n \choose v} \le {n \choose f/2 + 2},$$ and then we may use the estimate
$${n \choose f/2 + 2} \le \left( \frac{en}{f} \right)^{f/2+2},$$
valid for $f \ge 4$.

\begin{align}
\sum_{f=2M-2}^{n} \sum_{v = \sqrt{2f}}^{f/2 + 2} {n \choose v} Z(f,v) p^f
& \le
\sum_{f=2M-2}^{n} \sum_{v = \sqrt{2f}}^{f/2 + 2} \left( \frac{en}{f} \right)^{f/2+2} C_{\delta}^f f^{f/2} v^{\delta f} p^f,\\
& \le n^2 \sum_{f=2M-2}^{n} \sum_{v = \sqrt{2f}}^{f/2 + 2} \left( \left(\frac{en}{f} \right)^{1/2} C_{\delta} f^{1/2} v^{\delta} p \right)^f\\
& \le n^2 \sum_{f=2M-2}^{n} \sum_{v = \sqrt{2f}}^{f/2 + 2} \left( \left(en \right)^{1/2} C_{\delta} v^{\delta} n^{-1+ \alpha} \right)^f\\
& \le n^2 \sum_{f=2M-2}^{n} \sum_{v = \sqrt{2f}}^{f/2 + 2} \left( C'_{\delta} n^{\delta} n^{-1/2+ \alpha} \right)^f,\\
& \le n^3 \sum_{f=2M-2}^{n} \left( C'_{\delta} n^{\delta-1/2+ \alpha} \right)^f,
\end{align}
where $C'_{\delta}$ is a constant which only depends on $\delta$.

This sum can be bounded, term by term, by the infinite geometric series
$$a + ar + ar^2 + \dots = \frac{a}{1-r},$$
where $$a = n^3 \left( C'_{\delta} n^{\delta-1/2+ \alpha} \right)^{2M-2}$$
and $$r =  C'_{\delta} n^{\delta-1/2+ \alpha}.$$

We have chosen $\delta$ such that $$\delta \le \frac{1 - 2 \alpha }{4},$$
so
$$ \delta - \frac{1}{2} + \alpha \le \frac{ 2 \alpha - 1}{4} < 0.$$
We have chosen $M$ so that
$$2M  -2 \ge \frac{16}{1 - 2 \alpha} > 0.$$
So
$$(\delta-1/2+ \alpha)(2M- 2) \le -4,$$
and
$$ a= O \left(n^{-1} \right).$$
Since $ \delta-1/2+ \alpha < 0$, we also have that $r = o( 1)$.

Since $a \to 0$ and $r \to 0$ as $n \to \infty$, we also have that $a / (1-r) \to 0$, and the probability that there are any intermediate cycles tends to zero.

%

\subsection{Large cycles} Finally, we show that for any fixed $0 < \alpha < 1/2$ and $\epsilon > 0$,
$$\sum_{f=n}^{n^{2 - 2 \alpha - \epsilon }} \sum_{v = \sqrt{2f}}^{f/2 + 2} {n \choose v} Z(f,v) p^f \to 0,$$
as $n \to \infty$, so by the union bound, w.h.p.\ there are no large cycles.


We require the estimate $|Z(f,v)| \le C_{\delta}^f f^{f/2} v^{\delta v}$ again.
\begin{align*}
\sum_{f=n}^{n^{2 - 2 \alpha - \epsilon }} \sum_{v = \sqrt{2f}}^{f/2 + 2} {n \choose v} Z(f,v) p^f
& \le \sum_{f=n}^{n^{2 - 2 \alpha - \epsilon }} \sum_{v = \sqrt{2f}}^{f/2 + 2} {n \choose v} C_{\delta}^f f^{f/2} v^{\delta f} p^f\\
& \le \sum_{f=n}^{n^{2 - 2 \alpha - \epsilon }} C_{\delta}^f f^{f/2} p^f \sum_{v = \sqrt{2f}}^{f/2 + 2} {n \choose v} v^{\delta f},\\
& \le \sum_{f=n}^{n^{2 - 2 \alpha - \epsilon }} C_{\delta}^f f^{f/2} p^f 2^n f^{\delta f},\\
& \le \sum_{f=n}^{n^{2 - 2 \alpha - \epsilon }} \left(C_{\delta} f^{1/2 + \delta} p \right)^f 2^n,\\
& \le \sum_{f=n}^{n^{2 - 2 \alpha - \epsilon }} \left(C'_{\delta} f^{1/2 + \delta} n^{-1 + \alpha} \right)^f,\\
& \le \sum_{f=n}^{n^{2 - 2 \alpha - \epsilon }} \left(C'_{\delta} n^{(2 - 2 \alpha - \epsilon)(1/2 + \delta)} n^{-1 + \alpha} \right)^f.\\
\end{align*}

Since we chose $\delta > 0$ such that
$$\delta < \frac{\epsilon}{2(2 - 2 \alpha - \epsilon)},$$
we have
$$(2 - 2 \alpha - \epsilon)(1/2 + \delta) - 1 + \alpha < 0,$$
in which case we can bound the sum by a geometric series whose first term and ratio are tending to zero.

\subsection{Cycle counts from counting triangulated surfaces} \label{sec:cyc}

Recall that $Z(f,v)$ is the set of inclusion-minimal $2$-cycles on vertex set $[v] =\{ 1, 2, \dots, v \}$ with exactly $f$ faces. We bound $|Z(f,w)|$ by bounding $|T(f,w)|$, the number of combinatorial isomorphism types of connected triangulated surfaces on $f$ faces and $w$ vertices. Theorem 5.1 gives the following bound on $|T(f,w)|$: for any $\delta > 0$, there is a constant $C_\delta$ so that

\begin{equation} \label{5.1} |T(f,w) | \le C_\delta^f f^{f/2} w^{- (1 - \delta) w}. \end{equation}

Next we note that $|Z(f,v)|$ is related to $| T(f,w)|$ by the following inequality:

\begin{equation} \label{ZvsT} |Z(f,v)| \le \sum_{w \ge v} |T(f,w)| v^w. \end{equation}

We can see Inequality \ref{ZvsT} as follows. Any 2-cycle can be realized as the image of a simplicial map from a triangulated surface which is injective on 2-dimensional faces. If the 2-cycle is minimal, then the triangulated surface must be connected. For a fixed triangulated surface with $w$ vertices, the number of simplicial maps to the simplex with $v$ vertices is at most $v^w$.

Combining (\ref{5.1}) and (\ref{ZvsT}), we have

\begin{align*}
|Z(f,v)| & \le \sum_{w \ge v} |T(f,w)| v^w\\
&\le \sum_{w = v}^{f/2 + 2} \tilde{C}_{\delta}^f f^{f/2} w^{-(1-\delta) w} v^w\\
& = \tilde{C}_{\delta}^f f^{f/2} \sum_{w = v}^{f/2 + 2} (v / w^{1- \delta})^w\\
& \le \tilde{C}_{\delta}^f f^{f/2} (f/2+2) v^{\delta ( f/2 + 2)}\\
& \le C_{\delta}^f f^{f/2} v^{\delta f},
\end{align*}
for large enough $v$ and $f$, where $C_{\delta}$ is a constant which only depends on $\delta>0$.

\section{Counting triangulated surfaces} \label{sec:count}

Recall that $T(f,w)$ is the set of connected triangulated surfaces with $f$ faces and $w$ vertices, counted up to simplicial isomorphism. In this section, we prove an upper bound for $|T(f,w)|$.

\begin{thm} \label{triangcount} For every $\delta > 0$, there exists a constant $C_\delta$ so that for all $f, w$,

$$|T(f,w)| \le C_\delta^f f^{f/2} w^{-(1-\delta) w}. $$
\end{thm}

In order to estimate $|T(f,w)|$, we imagine starting with $f$ simplices and gluing together edges one at a time until we get a closed surface. This point of view was suggested by Brooks and Makover, \cite{BM}.

We make a formal definition of this gluing process. Let $\Delta_1, ..., \Delta_f$ be copies of the standard 2-simplex. Each of them has three edges, $\Delta_{j,a}$ with $a = 0, 1, 2$. A gluing story for these $f$ simplices is a sequence of $3f/2$ gluing maps $g_1, ..., g_{3f/2}$. Each gluing map $g_k$ is a simplicial isomorphism from an edge $\Delta_{j(k), a(k)}$ to another edge $\Delta_{j'(k), a'(k)}$. (A gluing map is allowed to map one edge of a simplex to another edge of the same simplex, but it is not allowed to map an edge to itself.) The maps $g_k$ make a gluing story if each edge is involved in exactly one gluing map $g_k$. We abbreviate a gluing story by $\vec g$, and we let $GS(f)$ denote the set of all gluing stories for $f$ 2-simplices.

It's straightforward to count the total number of gluing stories.

\begin{prop} \label{numbergluestory} $|GS(f)| =(3f)! 2^{3f/2}$.
\end{prop}
\begin{proof}
We first need to list the sequence of domains of $g_k$ and ranges of $g_k$. This amounts to listing the $3f$ edges of the triangles $\Delta_j$ in some order, so there are $(3f)!$ choices. After choosing the domain and range of each $g_k$, we have two choices for each map $g_k$, because there are two simplicial isomorphisms from one interval to another.
\end{proof}

Given a gluing story $\vec g$, we can define a 2-dimensional surface $X(\vec g)$ by starting with the $f$ simplices $\Delta_1, ..., \Delta_f$ and identifying points $p$ and $q$ if one of the gluing maps takes $p$ to $q$. The resulting object is not always a triangulated surface. It is a slightly more general object called a pseudomanifold. For example, $X(\vec g)$ could consist of two triangles, one on top of the other, with corresponding edges glued together. The resulting surface is homeomorphic to $S^2$, but it is not a triangulated surface. (Recall that a triangulated surface is a simplicial complex which is homeomorphic to a 2-dimensional manifold. This last example is not a simplicial complex.) We will discuss pseudomanifolds more below.

Here is an outline of the proof of Theorem \ref{triangcount}. Each surface $X \in T(f,w)$ is simplically isomorphic to $X(\vec g)$ for some $g \in GS(f)$. In fact, each surface $X \in T(f,w)$ can be realized by many different gluing stories, and we have to estimate the size of this overcount.

\begin{lem} \label{overcount} There is a constant $C > 0$ so that the following holds. For every $X \in T(f,w)$, there are at least $C^{-f} f^{5f/2}$ gluing stories $\vec g$ so that $X(\vec g)$ is simplicially isomorphic to $X$.
\end{lem}

This lemma is similar to results in \cite{BM}, but we will give a self-contained proof.

Next we will estimate the number of gluing stories that produce triangulated surfaces with $w$ vertices. This estimate is the new ingredient in the proof of Theorem \ref{triangcount}.

\begin{lem} \label{gluecount} For any $\delta > 0$, there is a constant $C_\delta$ so that for any $f,w$, the number of gluing stories $\vec g \in GS(f)$ on $f$ so that $X(\vec g)$ is a connected triangulated surface with $w$ vertices is $\le C_\delta^f f^{3f} w^{-(1-\delta)w}$.
\end{lem}

\noindent Combining Lemma \ref{overcount} and Lemma \ref{gluecount} gives Theorem \ref{triangcount}.

If we compare Proposition \ref{numbergluestory} and Lemma \ref{gluecount}, we see that the fraction of gluing stories $\vec g \in GS(f)$ so that $X(\vec g)$ is a triangulated surface with $w$ vertices is at most roughly $w^{-w}$. In particular, gluing stories that produce triangulated surfaces with many vertices are quite rare.

Here is the intuition behind our argument. Imagine that we carry out the gluings one map at a time. Before the first gluing, we have $f$ disjoint triangles. At each step of the process, we glue together two of the edges in the boundary. After performing $k$ gluings, we have a surface with boundary, called $X_k(\vec g)$. After using all $3f/2$ gluing maps, we have $X(\vec g)$. A vertex of $X_k(\vec g)$ is called a {\it boundary vertex} if it lies in the boundary of $X_k(\vec g)$, and an {\it internal vertex} otherwise. Let $V_{int} (X_k)$ be the set of internal vertices of $X_k (\vec g)$. Since $X(\vec g)$ has no boundary, every vertex of $X(\vec g)$ is internal. On the other hand, $X_0(\vec g)$ has zero internal vertices. The only way that $X_{k+1}(\vec g)$ can have more internal vertices than $X_{k}(\vec g)$ is if $g_{k+1}$ glues together two edges of $\partial X_{k}(\vec g)$ that share a vertex. This is a rare event. The number of edges in $\partial X_{k}(\vec g)$ is $3f - 2 k$. Each boundary edge shares a vertex with at most two other edges. If we randomly pick two edges of $\partial X_k(\vec g)$, the probability that they share a vertex is on the order of $(3f - 2 k)^{-1}$. This suggests that gluing stories that produce many vertices are quite rare.

In our proof, we turn this intuition into a precise estimate. Getting the quantitative result that we need from this approach was technically tricky, and we discuss this more below. For example, it is much easier to prove that $|T(f,w)| \le C^f f^{f/2} w^{- w /2}$. However, this weaker estimate is too weak for our applications in the paper, and we need to do some work to get the best estimate that we can.

In subsection \ref{subsecovercount} we estimate the overcounting and prove Lemma \ref{overcount}. In subsection \ref{subsecmanyvertices} we prove Lemma \ref{gluecount}.

\subsection{Gluing stories for a given triangulated surface} \label{subsecovercount}

In this section, we prove Lemma \ref{overcount}. If $X$ is a connected triangulated surface with $f$ faces, then we have to prove that there are at least $C^{-f} f^{5f/2}$ different gluing stories $\vec g \in GS(f)$ so that $X(\vec g)$ is simplicially isomorphic to $X$.

The simplicial automorphisms of a triangulated surface $X$ will be involved in the proof. Let $Aut(X)$ be the group of simplicial automorphisms of $X$. We will need the following well-known estimate.

\begin{lem} \label{autcount} If $X$ is a connected triangulated surface with $f$ faces, then $|Aut(X)| \le 6 f. $ \end{lem}

\begin{proof} Let $\Delta_1$ be a face in $X$. There are $6 f$ simplicial isomorphisms from $\Delta_1$ to one of the faces of $X$. We claim that each of these maps extends to at most one simplicial isomorphism from $X$ to itself.

Let $\phi_1, \phi_2: X \rightarrow X$ be simplicial isomorphisms, and suppose that $\phi_1$ and $\phi_2$ agree on a 2-face $\Delta$. We say that two faces $\Delta, \Delta' \subset X$ are adjacent if they have a common edge. We will show that $\phi_1$ and $\phi_2$ also agree on each simplex adjacent to $\Delta$. Suppose that $\Delta'$ is adjacent to $\Delta$ and that $e = \Delta \cap \Delta'$ is the edge that they both contain. We know that $\phi_1$ and $\phi_2$ agree on the edge $e$. Now $\phi_1(e) = \phi_2(e)$ is an edge of $X$, and so it lies in exactly two faces of $X$. We know that $\phi_1(\Delta) = \phi_2(\Delta)$ is one of these faces. Since $\phi_1$ and $\phi_2$ are both simplicial isomorphisms, they must both map $\Delta$ to the other of these faces. Now $\phi_1$ and $\phi_2$ agree on the edge $e$, and so they must also agree on the third vertex of $\Delta'$, and so $\phi_1$ and $\phi_2$ agree on $\Delta'$.

Now suppose $\phi_1, \phi_2$ are simplicial isomorphisms $X \rightarrow X$ that agree on $\Delta_1$. By the result of the last paragraph, they must agree on all the simplices that are adjacent to $\Delta_1$. Iterating the argument, they must agree on all the simplices adjacent to these simplices. Since $X$ is connected, by iterating this argument, we see that $\phi_1$ and $\phi_2$ must agree everywhere. \end{proof}

Now we turn to the proof of Lemma \ref{overcount}.

\begin{proof} Let $X$ be a connected triangulated surface with $f$ faces. First we check that there is a gluing story $\vec g \in GS(f)$ so that $X$ is simplicially isomorphic to $X(\vec g)$. Number the faces of $X$ as $F_1$, ..., $F_f$. For each face $F_j$, pick an identification of the face with a standard simplex $\Delta_j$. Now number the edges of $X$ from 1 to $3f/2$. Suppose that $e_k$ is the $k^{th}$ edge, and that it lies in faces $F_j(k)$ and $F_{j'(k)}$.  The edge $e_k$ corresponds to an edge $\Delta_{j(k), a(k)} \subset \Delta_{j(k)}$ and to an edge $\Delta_{j'(k), a'(k)} \subset \Delta_{j'(k)}$. Since these two edges are identified in $X$, we get a simplicial isomorphism $g_k: \Delta_{j(k), a(k)} \rightarrow \Delta_{j'(k), a'(k)}$. The sequence of $g_k$ make a gluing story $\vec g \in GS(f)$ and $X(\vec g)$ is simplically isomorphic to $X$.

Given a gluing story $\vec g$, we now produce many other gluing stories that lead to the same surface $X(\vec g)$. These other gluing stories come from relabelling the characters in the initial gluing story $\vec g$. There are two different kinds of relabelling that we can do. We can reorder the gluing maps $g_k$. In other words, we can consider a gluing story with the same set of gluing maps in a different order. Reordering the gluing maps $g_k$ defines an action of the symmetric group $S_{3f/2}$ on $GS(f)$. We can also relabel the simplices $\Delta_1, ..., \Delta_f$. This relabelling gives an action of $S_f$ on $GS(f)$. These two actions commute, and so we get an action of $S_{3f/2} \times S_f$ on $GS(f)$. If $\vec g$ and $\vec g'$ are in the same orbit of this action, then $X(\vec g)$ and $X(\vec g')$ are simplicially isomorphic.

This action is not necessarily free. We note that $|S_{3f/2} \times S_f| = (3f/2)! f! \ge C^{-f} f^{5f/2}$. Recall that we found an element $\vec g \in GS(f)$ so that $X( \vec g) $ is simplicially isomorphic to our given $X$. For every $\vec h$ in the $S_{3f/2} \times S_f$-orbit of $\vec g$, $X(\vec h)$ is also simplicially isomorphic to $X$. The size of this orbit is
 $|S_{3f/2} \times S_f| / |Stab(\vec g)|$. (Here $Stab(\vec g) \subset S_{3f/2} \times S_f$ is the stabilizer subgroup: the set of group elements $\psi \in S_{3f/2} \times S_f$ so that $\psi(\vec g) = \vec g$.)

To finish the proof, we will check that $|Stab(\vec g)| \le |Aut(X)|$. By Lemma \ref{autcount} above, $|Aut(X)| \le 6f$.
To see that $|Stab(\vec g)| \le |Aut(X)|$, we will construct a natural injection $Stab(\vec g) \rightarrow Aut ( X(\vec g))$.

Suppose that $(\psi_1, \psi_2) \in Stab(\vec g)$, where $\psi_1 \in S_{3f/2}$ and $\psi_2 \in S_f$. We use $\psi_2$ to define a simplicial map from $X(\vec g)$ to itself. The map sends $\Delta_j$ to $\Delta_{\psi_2(j)}$ by the identity (remember, the $\Delta_j$ are all copies of the standard 2-simplex).
We have to check that this map respects all of the gluings. But if $g_k$ glues $\Delta_{j(k), a(k)}$ to $\Delta_{j'(k), a'(k)}$, then $g_{\psi_1(k)}$ glues $\Delta_{\psi_2(j(k)), a(k)}$ to $\Delta_{\psi_2(j'(k)), a'(k)}$ by the same simplicial isomorphism. Therefore, our map does respect all the gluings, and it gives a simplicial map. Applying the same construction with $(\psi_1^{-1}, \psi_2^{-1})$ we get an inverse simplicial map, so $(\psi_1, \psi_2)$ was mapped to a simpicial isomorphism $X \rightarrow X$.

We now have a group homomorphism $Stab(\vec g) \rightarrow Aut( X(g))$. We next show that this homomorphism is injective. Suppose that $(\psi_1, \psi_2) \in Stab(\vec g)$ corresponds to the identity map $X \rightarrow X$. By construction, we see that $\psi_2$ is the identity. But just reordering the gluing maps will produce a different gluing story unless $\psi_1$ is the identity also. So the kernel of the homomorphism is the identity. \end{proof}

\subsection{Gluing stories with many vertices} \label{subsecmanyvertices}

In this section, we prove Lemma \ref{gluecount}. Recall that Lemma \ref{gluecount} gives a bound on the number of gluing stories $\vec g \in GS(f)$ so that $X(\vec g)$ is a connected triangulated surface with $w$ vertices.

Let $\vec g \in GS(f)$ be a gluing story. For $0 \le k \le 3f/2$, let $X_k(\vec g)$ be the space formed from the simplices $\Delta_1, ..., \Delta_f$ by making identifications using the first $k$ gluing maps, $g_1, ..., g_k$. The space $X_0(\vec g)$ is just a disjoint union of $f$ simplices, and $X_{3f/2}(\vec g) = X(\vec g)$.
To help understand the number of vertices in $X(\vec g)$, we will consider the vertices of $X_k(\vec g)$ for every $k$ and keep track of how they change as $k$ increases.

Before turning to the proof, we should talk briefly about what type of object $X_k(\vec g)$ is. We observed above that for a general gluing story $\vec g \in GS(f)$, the space $X(\vec g)$ may not be a triangulated surface. In Lemma \ref{gluecount}, we can restrict attention to gluing stories so that $X(\vec g)$ is a triangulated surface.  But even if $X(\vec g)$ is a triangulated surface, the spaces $X_k(\vec g)$ may not all be triangulated surfaces with boundary. For example, consider the second-to-last space $X_{(3f/2) - 1}(\vec g)$. The boundary of this space must consist of two edges, and the last gluing map $g_{3f/2}$ glues together these two edges. These two edges must form a loop of length 2, which means that they have the same boundary vertices. Therefore, $X_{(3f/2) - 1}(\vec g)$ is not a simplicial complex, and so it is certainly not a triangulated surface with boundary.

The spaces $X_k(\vec g)$ are always pseudomanifolds with boundary. In Subsection \ref{subsecbackpseudo}, we give an appendix recalling all the definitions and facts about pseudomanifolds with boundary that we need.

For a general gluing story $\vec g$, the boundary of $X_k(\vec g)$ consists of a disjoint union of loops, and each loop consists of at least one edge. These loops are called the connected components of the boundary of $X_k(\vec g)$. For general $\vec g$, the boundary of $X_k(\vec g)$ may contain a loop with only one edge. However, if $X(\vec g)$ is a triangulated surface, then each component of the boundary of $X_k(\vec g)$ contains at least two edges. We can see this as follows. Since $X(\vec g)$ is a triangulated surface, in particular a simplicial complex, there is no edge in $X(\vec g)$ from a vertex to itself. But then $X_k(\vec g)$ cannot contain any edge from a vertex to itself either. For more details, see Subsection \ref{subsecbackpseudo}.

The pseudomanifold with boundary $X_{k+1}$ is formed from $X_k$ by gluing together two of the boundary edges of $X_k$ by the gluing map $g_{k+1}$. To prove Lemma \ref{gluecount}, we will pay attention to whether we glue together ``nearby'' edges or ''far apart'' edges. We say that two edges of $\partial X_k$ are adjacent if they share a vertex. We say that edges $e$ and $e'$ in $\partial X_k$ are $D$-near if there is a string of adjacent edges in $\partial X_k$, $e = e_0$ adjacent to $e_1$, $e_j$ adjacent to $e_{j+1}$ for $1 \leq j < D$, and $e_D = e'$.
We will show that if $X(\vec g)$ is a triangulated surface with many vertices, then many of the gluing maps $g_k$ must glue together nearby edges. Here is a lemma that makes this precise.

\begin{lem} \label{nearmoves} For any $\delta > 0$, there is a $D_\delta$ so that the following holds. Suppose that $\vec g \in GS(f)$ is a gluing story and that $X(\vec g)$ is a triangulated surface with $V(X)$ vertices and $H(X)$ connected components. Suppose that $N$ of the gluing steps $g_k$ glue together edges that are $D_\delta$-near.
Then

$$ V(X) \le (1 + \delta) N + H(X). $$

In particular, if $X(\vec g)$ is a connected triangulated surface, then

$$ V(X) \le (1 + \delta) N + 1. $$

\end{lem}

First we prove that Lemma \ref{nearmoves} implies Lemma \ref{gluecount}.

\begin{proof}[Proof of Lemma \ref{gluecount} assuming Lemma \ref{nearmoves}.]
 The point of the proof is that there are not very many ways to glue two nearby edges of $\partial X_k$. For any $X_k$, there are at most
$12 D_\delta f$ gluing maps $g_{k+1}$ between edges that are $D_\delta$ near. This is because there are at most $3f$ edges for the domain of $g_{k+1}$. Then there are at most $2 D_\delta$ edges that are $D_\delta$ near to the first edge. Then there are at most two gluings from the first edge to the second edge. There are at most $2 (3f)^2 = 18 f^2$ possible gluing maps $g_{k+1}$. Suppressing the constants, the number of possible gluing maps between nearby edges is $\le C_\delta f$, and the total number of possible gluing maps is $\le C f^2$.

The number of gluing stories in $GS(f)$ with exactly $N$ $D_\delta$-near moves is at most

\begin{align*}
{3f/2 \choose N} (C_\delta f)^N (C f^2)^{(3f/2) - N} &\le (3f/2)^N (N!)^{-1} (C_\delta f)^N (C f^2)^{(3f/2) - N} \\
& \le C_\delta^f f^{3f} (N!)^{-1} \\
&\le C_\delta^f f^{3f} N^{-N}. \\
\end{align*}

In other words:

\begin{equation} \label{nearmovesbound}
| \{ \vec g \in GS(f) | \vec g \textrm{ has exactly } N D_\delta\textrm{-near moves} \} | \le C_\delta^f f^{3f} N^{-N}.
\end{equation}

Let $GS_w(f) \subset GS(f)$ be the set of gluing stories $\vec g \in GS(f)$ so that $X(\vec g)$ is a connected triangulated surface with $w$ vertices. Lemma \ref{nearmoves} implies that for any $\vec g \in GS_w(f)$, the number of $D_\delta$-near gluing maps in $\vec g$ is at least $\frac{w-1}{1+\delta}$. By equation \ref{nearmovesbound}, we see that

$$ |GS_w(f)| \le \sum_{N \ge \frac{w-1}{1+ \delta}} C_\delta^f f^{3f} N^{-N} \le C_\delta^f f^{3f} w^{-\frac{w}{1 + \delta}}.$$

Since $\delta > 0$ is arbitrary, this proves Lemma \ref{gluecount}. \end{proof}

Remark. We will give below a short proof that $V(X) \le 2 N$. This bound is not strong enough to prove Lemma \ref{gluecount}. Using the bound $V(X) \le 2N$ in place of Lemma \ref{nearmoves} in the argument above leads to the estimate $|T(f,w)| \le C^f f^{f/2} w^{-w/2}$. This bound is much weaker than Theorem \ref{triangcount} when $w$ is large. In Lemma \ref{nearmoves}, in the bound $V(X) \le (1 + \delta) N + H(X)$, it is important to get the right constant in front of the $N$.

Next we discuss why the number of connected components plays a role in Lemma \ref{nearmoves}. Consider gluing together a tetrahedron from four faces, $\Delta_0, \Delta_1, \Delta_2, \Delta_3$. For the first three moves, we attach $\Delta_1, \Delta_2,$ and $\Delta_3$ to the three edges of $\Delta_0$. These steps are gluings between edges in different components. Then we do three more gluings and get a tetrahedron. The last three gluings connect nearby edges. If $f$ is a multiple of 4, we can repeat this procedure $f/4$ times to get $f/4$ tetrahedra. The $f/4$ tetrahedra have $V = f$ vertices. In this story, the number of gluings between nearby edges is $N = 3f/4$. So in this example $V = (4/3) N$, which is too large. But this example also has $H = f/4$ connected components. So we see that $V = N + H$. This example shows that we need to include the number of connected components in our estimate.

Next we classify gluing moves according to how they affect the number of internal vertices of $X_k$. Before we do this, it is helpful to remark that if $X(\vec g)$ is a triangulated surface, then every edge of every $X_k$ must have two different endpoints.

Suppose that $g_{k+1}$ is a gluing map from $e_1$ to $e_2$. If $e_1$ and $e_2$ share no vertices, we say that $g_{k+1}$ is a gluing of type $A$. In this case, the gluing $g_{k+1}$ creates no new internal vertices: $V_{int}(X_{k+1}) = V_{int}(X_k)$. If $e_1$ and $e_2$ share exactly one vertex, we say that $g_{k+1}$ is a gluing of type $B$. In this case, $g_{k+1}$ creates one new internal vertex: $V_{int}(X_{k+1}) = V_{int}(X_k) + 1$. It can also happen that $e_1$ and $e_2$ share two vertices! In other words, $e_1$ and $e_2$ are both edges between the same two vertices $v, v'$. In this case, we say that $g_{k+1}$ is a gluing of type $C$, and we note that $V_{int}(X_{k+1}) = V_{int}(X_k) + 2$.

The gluing maps of type $C$ are crucial in our story, so we take a moment to describe an example. Such examples can occur even when $X(\vec g)$ is a triangulated surface. Suppose that $f=6$ so that the gluing story has $3f/2 = 9$ moves.  The boundary of $X_7$ has four edges. Suppose that $X_{7}$ has one boundary component which consists of four edges. (This is not difficult to arrange.) Next suppose that $g_8$ glues together two adjacent edges of this boundary.
So $g_8$ is a gluing map of type $B$. Now $X_8$ has one boundary component consisting of two edges. The gluing map $g_9$ must glue together these two edges. So $g_9$ is a gluing map of type $C$. Notice that the number of vertices of $X(\vec g)$ is $V_{int}(X_9) = V_{int}(X_8) + 2$.

The number of vertices of $X(\vec g)$ is $V_{int}(X_{3f/2}) = B + 2C$, where we write $B$ for the number of type $B$ gluing maps in $\vec g$, and similarly for $C$. Gluing maps of type $B$ or $C$ are 1-near, and so we see that
$V(X) \le 2 N$. As we discussed above, this estimate is not strong enough to prove Theorem \ref{triangcount}. We need to be more careful in how we deal with gluing maps of type $C$.

Now we begin the rigorous proof of Lemma \ref{nearmoves}:

\begin{proof}[Proof of Lemma \ref{nearmoves}.] Here is the frame of the proof. We set $D_\delta = 10^{1 / \delta}$. We will define some function $F(X_k)$ and check the following properties.

\begin{enumerate}

\item $F(X_0) = 0$.

\item $F(X_{3f/2}) = V(X) - H(X)$.

\item If $g_{k+1}$ glues together two $D_\delta$-near edges, then

$$F(X_{k+1}) \le F(X_k) + 1 +\delta.$$

If $g_{k+1}$ glues together two edges which are not $D_\delta$-near, then

$$F(X_{k+1}) \le F(X_k).$$

\end{enumerate}

Given these properties, it is easy to finish the proof of the lemma.  By Property 2, $V(X) - H(X) = F(X_{3f/2})$. By Property 3, $F(X_{3f/2}) \le (1 + \delta) N + F(X_0)$. By Property 1, this is equal to $(1 + \delta)N$. So all together, we have $V(X) - H(X) \le (1+\delta) N$. Hence $V(X) \le (1 + \delta) N + H(X)$.

The main difficulty is to craft a function $F$ that obeys these properties. The function $F(X_k)$ will be the number of internal vertices of $X_k$ plus some other terms. Before writing down the detailed formula, we try to motivate these other terms. A gluing of type $C$ increases the number of internal vertices by 2, and it gets rid of a boundary component of length 2. Since $F$ is only allowed to increase by $1+\delta$, we decide that a boundary component of length 2 should contribute approximately $1 - \delta$ to $F$. Now a gluing of type $C$ only increases $F$ by $1 + \delta$.

But this patch creates new issues. For instance, if we glue together two edges in a boundary component of length 6, we can get two boundary components of length 2. The two boundary components of length 2 contribute approximately $2 - 2 \delta$ to $F$. Since $F$ is only allowed to increase by $1+\delta$, we decide that a boundary component of length 6 should contribute approximately $1 - 3 \delta$ to $F$. In general, a boundary component of length $l$ contributes $\beta(l)$ to $F$, where $\beta(l)$ decreases slowly to zero. Note that if we glue together two edges that are $D_\delta$ far apart, we can create a new boundary component of length $D_\delta$. When we glue far apart edges, $F(X_k)$ cannot increase at all, so we have to arrange that $\beta(l) = 0$ for $l \ge D_\delta$.

But this scheme leads to another issue. The initial configuration $X_0$ has many boundary components of length 3. We want $F(X_0) = 0$. To fix this problem, we only count some of the boundary components. More precisely, for each component $X_k' \subset X_k$, we don't include the boundary contribution from the longest boundary component of $X_k'$. In particular, each component of $X_0$ has only a single boundary component and so the boundary contribution of $X_0$ is zero, as desired.

This modification creates yet another small issue. Suppose that we do a gluing move of type $C$ on a boundary component of length 2 which is the only boundary component of some component $X_k' \subset X_k$. The boundary component of length 2 no longer contributes to $F(X_k)$, and we still create two new internal vertices. In this situation, we also create a new closed connected component of $X_{k+1}$. (A connected component $X_k' \subset X_k$ is closed if it has no boundary.) We decide that each closed component of $X_k$ contributes $-1$ to $F(X_k)$. In this situation, the number of internal vertices goes up by two, but the number of closed components goes up by 1, and so $F$ increases by only 1.

With this motivation, we are ready to give the precise definition of $F(X_k)$ and check all of the properties.
The function $F(X_k)$ is a sum of three terms:

$$ F(X_k) = V_{int} (X_k) - H_{cl}(X_k) + B(X_k), $$

\noindent The terms are as follows. $V_{int}$ is the number of interior vertices of $X_k$. $H_{cl}(X_k)$ is the number of closed components of $X_k$: the number of components of $X_k$ which are pseudomanifolds without boundary. And $B(X_k)$ is a boundary term involving the lengths of the boundary components of $X_k$.

For $l \ge 1$, we define $\beta(l) = \max (0, 1 - \delta \log_{10} l)$. We have $\beta(l) \ge 0$, with $\beta(l) = 0$ for all $l \ge D_\delta = 10^{1/\delta}$. We see that $\beta(1) = 1$, and we note that $\beta$ is decreasing.

For a connected component $Y \subset \partial X_k$, we let $l(Y)$ denote the length of $Y$ (i.e. the number of edges in $Y$).  For a connected component $X_k' \subset X_k$, we define

$$ l_{max} (X_k') := \max_{Y \textrm{ a conn. compon. of } X_k'} l(Y). $$

For a connected component $X_k' \subset X_k$, we define $B(X_k')$ as follows:

$$ B(X_k') = \left( \sum_{Y \textrm{ a conn. compon. of } X_k'} \beta( l(Y)) \right) - \beta( l_{max}(X_k') ). $$

Finally, we define $B(X_k)$ as a sum of contributions from the connected components:

$$ B(X_k) := \sum_{X_k' \textrm{ a conn. compon. of } X_k} B(X_k'). $$

This finishes the definition of $F(X_k)$, and now we have to check Properties 1-3.

{\bf Property 1.} We know that $X_0$ is a disjoint union of $f$ 2-simplices. It has no interior vertices. It has no closed components. Each component of $X_0$ has a single boundary component of length 3, and so $B(X_0) = 0$. Therefore, $F(X_0) = 0$ proving Property 1.

{\bf Property 2.} It's also easy to analyze $F(X_{3f/2})$. We know that $X_{3f/2} = X (\vec g)$ has no boundary, so the complicated term $B(X_{3f/2})$ vanishes. That leaves $F(X_{3f/2}) = V_{int} (X_{3f/2}) - H_{cl}( X_{3f/2})$. Since $X_{3f/2}$ has no boundary, all its vertices are interior vertices, and all its connected components are closed. Therefore, $F(X_{3f/2}) = V(X) - H(X)$, proving Property 2.

{\bf Property 3.} We have to compare $F(X_{k+1})$ and $F(X_k)$.
We consider separately the cases that $g_{k+1}$ has type $A, B$, or $C$. We begin with type $C$, because it plays such an important role in the problem.

Suppose that $g_{k+1}$ has type $C$. By the definition of type C, the map $g_{k+1}$ glues together two edges that share two vertices. A map of type $C$ is 1-near, so we have to show that $F(X_{k+1}) \le F(X_k) + 1 + \delta$. The map $g_{k+1}$ creates two new interior vertices: $V_{int}(X_{k+1}) = V_{int} (X_k) + 2$. The two edges that are glued together by $g_{k+1}$ form a boundary component $Y$ of length $2$ in $\partial X_k$. Suppose that $Y \subset \partial X_k'$ for a connected component $X_k' \subset X_k$. Now we consider two cases.

\begin{itemize}

\item Suppose that $Y$ is the whole boundary of $X_k'$. In this case, the gluing step does not change the boundary term: $B(X_{k+1}) = B(X_k)$.  Also, the number of closed components increases by 1: $H_{cl}(X_{k+1}) = H_{cl}(X_k) + 1$. Assembling all the terms, we get $F(X_{k+1}) = F(X_k) + 1$.

\item Suppose that $X_k'$ has other boundary components. In this case, the gluing step reduces the boundary term by $\beta(2)$: $B(X_{k+1}) = B(X_k) - \beta(2)$. The number of closed components remains the same. Assembling all the terms, we get $F(X_{k+1}) = F(X_k) + 2 - \beta(2) = F(X_k) + 1 + \delta \log_{10}(2)$.

\end{itemize}

Next, we suppose that $g_{k+1}$ has type $B$. The map glues together two adjacent edges in a boundary component of some length $l \ge 3$. A gluing of type $B$ is 1-near, so again we have to prove that $F(X_{k+1}) \le F(X_k) + 1 + \delta$. The gluing creates one new interior vertex. It does not change the number of closed components. The boundary term may increase by at most $\beta(l-2) - \beta(l) \le \delta \log_{10} \frac{l}{l-2} \le \delta \log_{10} 3 \le \delta$. Therefore, $F(X_{k+1}) \le F(X_k) + 1 + \delta$.

Finally we suppose that $g_{k+1}$ has type $A$. In this case, the number of interior vertices and the number of closed components do not change, so we only need to analyze the boundary term.
A gluing map of type $A$ may or may not be $D_\delta$-near. The type A case has a number of sub-cases as follows.

\begin{enumerate}

\item The map $g_{k+1}$ glues together two edges in the same component of $\partial X_k$.

\item The map $g_{k+1}$ glues together two edges in different components of $\partial X_k$, but in the same component of $X_k$.

\item The map $g_{k+1}$ glues together two edges in different components of $X_k$.

\end{enumerate}

We first consider Case 1: the map $g_{k+1}$ glues together two edges in the same component of $\partial X_k$. Suppose that this boundary component has length $l \ge 3$. After the gluing, depending on the orientation of the gluing, the boundary component of length $l$ either becomes two boundary components of lengths $l_1, l_2 \ge 2$ where $l_1 + l_2 +2 = l$, or else it becomes one boundary component of length $l-2$. The most interesting case is when the boundary component splits into two components of lengths $l_1$, $l_2$. We discuss this case first. We begin by noting that $l \ge \max(l_1, l_2)$.

We let $X_k'$ be the component of $X_k$ that contains the edges where $g_{k+1}$ acts. We let $X'_{k+1}$ be the corresponding component of $X_{k+1}$. Now the change in the boundary term is

\begin{equation} \label{boundchange1}
B(X_{k+1}) - B(X_k) = - \beta(l) + \beta(l_1) + \beta(l_2) + \beta (l_{max}(X'_k)) - \beta( l_{max} (X'_{k+1}). \end{equation}

If the gluing $g_{k+1}$ is $D_\delta$-far, then $l_1, l_2, l$ are all at least $D_\delta$, and hence $l_{max}(X'_k)$ and $l_{max}(X'_{k+1})$ are also at least $D_\delta$. Therefore, all the terms on the right-hand side of equation \ref{boundchange1} vanish.

If the gluing $g_{k+1}$ is $D_\delta$-near, then we note that $l_{max}( X'_{k+1}) \le l_{max}(X'_k)$. Therefore,

$$ B(X_{k+1}) - B(X_k) \le - \beta(l) + \beta(l_1) + \beta(l_2). $$

This increase is acceptable by the following lemma:

\begin{lem} Suppose that $l_1, l_2 \ge 2$ and $l \ge 3$ are integers with $l_1 + l_2 + 2 = l$. Then
$\beta(l_1) + \beta(l_2) - \beta(l) \le 1+ \delta$.
\end{lem}

\begin{proof} We can assume that $l_1 \le l_2$. Since $l_1, l_2 \ge 2$, we have $l \ge l_1, l_2$. If $\beta(l_2) = 0$, then $\beta(l_1) + \beta(l_2) - \beta(l) \le \beta(l_1) \le 1$. So we can assume that $\beta(l_1)$ and
$\beta(l_2)$ are positive. Then we get

\begin{align*}
\beta(l_1) + \beta(l_2) - \beta(l) &\le (1 - \delta \log_{10}(l_1) ) + ( 1 - \delta \log_{10}(l_2)) - (1 - \delta \log_{10} (l)) \\
& = 1 + \delta \left( \log_{10} (l) - \log_{10}(l_1) - \log_{10}(l_2) \right).\\
\end{align*}

It suffices to prove that the expression in parentheses is $\le 1$.
Since $2 \le l_1 \le l_2$, we have $l_1 l_2 \ge l_1 + l_2$ and so
\begin{align*}
\log_{10} (l_1) + \log_{10}(l_2) & \ge \log_{10} (l_1 + l_2) \\
& = \log_{10} ( l-2).
\end{align*}
So we have
\begin{align*}
\log_{10} (l) - \log_{10}(l_1) - \log_{10}(l_2) & \le \log_{10}(l) - \log_{10}(l-2) \\
&= \log_{10} (\frac{l}{l-2})\\
&\le \log_{10} 3 \\
& \le 1.
\end{align*}
\end{proof}

Now we turn to the simpler possiblity that $g_{k+1}$ glues together two edges of a boundary component of length $l$, turning it into a single component of length $l-2$. (The orientation of the gluing map determines whether the boundary component of length $l$ splits into two boundary components or remains a single connected component of the boundary of $X_{k+1}$.)
In this case,

\begin{align*}
B(X_{k+1}) - B(X_k) &= - \beta(l) + \beta(l-2) + \beta (l_{max}(X'_k)) - \beta( l_{max} (X'_{k+1}) \\
& \le - \beta(l) + \beta(l-2).\\
\end{align*}

If $g_{k+1}$ is $D_\delta$-far, then all the terms vanish. If $g_{k+1}$ is $D_\delta$-near, then
$- \beta(l) + \beta(l-2) \le \delta \log_{10} \frac{l}{l-2} \le \delta$.

This finishes the analysis of Case 1, and now we turn to Case 2: the map $g_{k+1}$ glues together two edges in different components of $\partial X_k$, but in the same component $X_k' \subset X_k$.
Suppose that the two components in $\partial X_k'$ have lengths $l_1, l_2$ and the new component in $\partial X_{k+1}'$ has length $l_3 = l_1 + l_2 - 2 \ge \max(l_1, l_2)$. Note that in Case 2, the gluing map is automatically $D_\delta$ far, and so we have to prove that $B(X_{k+1}) \le B(X_k)$.
We expand

\begin{equation} \label{boundchange2} B(X_{k+1}) - B(X_k) = - \beta(l_1) - \beta(l_2) + \beta(l_3) + \beta( l_{max} (X_k') ) - \beta ( l_{max} (X_{k+1}') ). \end{equation}

We note that

$$ l_{max} (X_{k+1}') = \max( l_{max} (X_k'), l_3 ). $$

In the first case, the two $l_{max}$ terms cancel in equation \ref{boundchange2}, leaving

$$ B(X_{k+1}) - B(X_k) = - \beta(l_1) - \beta(l_2) + \beta(l_3) \le - \beta(l_1) \le 0. $$

In the second case, the $l_{max} (X_{k+1}')$ term and the $l_3$ term cancel in equation \ref{boundchange2} leaving

$$ B(X_{k+1}) - B(X_k) = - \beta(l_1) - \beta(l_2) + \beta( l_{max} (X_k') ) \le - \beta(l_1) \le 0.$$

This finishes the analysis of Case 2, and now we turn to Case 3: the map $g_{k+1}$ glues together two edges in different components of $X_k$, say $X_k'$ and $X_k''$. The map $g_{k+1}$ glues together an edge from a component of $X_k'$ with length $l_1$ and an edge from a component of $X_k''$ with length $l_2$. After the gluing, $X'_{k}$ and $X''_{k}$ have merged into one component $X_{k+1}''' \subset X_{k+1}$, and the boundary components of lengths $l_1$ and $l_2$ have merged into one component of length $l_3 = l_1 + l_2 -2$. We note as above that $l_3 \ge \max( l_1, l_2)$.
In Case 3, the gluing map $g_{k+1}$ is automatically $D_\delta$-far, and so we have to prove that $B(X_{k+1}) \le B(X_k)$.

We expand $B(X_{k+1}) - B(X_k)$ to get

\begin{equation} \label{boundchange3} - \beta(l_1) - \beta(l_2) + \beta(l_3) + \beta( l_{max} (X_k') ) + \beta (l_{max} (X_k'')) - \beta ( l_{max} (X'''_{k+1}) ). \end{equation}

We note that

$$ l_{max} (X'''_{k+1}) = \max \left( l_{max} (X_k'), l_{max} (X_k''), l_3 \right). $$

The first two cases are equivalent. If $l_{max} (X'''_{k+1}) = l_{max} (X'_k)$, then those two terms cancel from equation \ref{boundchange3}, leaving $B(X_{k+1}) - B(X_k) = $

$$ - \beta(l_1) - \beta(l_2) + \beta(l_3) + \beta (l_{max} (X_k'')) \le - \beta(l_1) + \beta(l_3) \le 0. $$

On the hand, if $l_{max} (X'''_{k+1}) = l_3$, then those two terms cancel from equation \ref{boundchange3}, leaving $B(X_{k+1}) - B(X_k) = $

$$ - \beta(l_1) + \beta( l_{max} (X_k') ) - \beta(l_2) + \beta (l_{max} (X_k'')) \le 0. $$

This finishes the analysis of Case 3. We have now checked Properties 1-3, finishing the proof of Lemma \ref{nearmoves}. \end{proof}

There are a number of open questions about counting surfaces with various restrictions related to the material in this section. First of all, it would be interesting to find upper and lower bounds for $|T(f,w)|$ that are as close together as possible. It would also be interesting to estimate the number of connected pseudomanifolds with $f$ faces and $w$ vertices, up to combinatorial equivalence. Finally, it would be interesting to consider generalizations to higher dimensions. There are several variations in higher dimensions. From the point of view of studying the $d$-dimensional girths of $d$-dimensional complexes, it would be helpful to estimate the number of connected $d$-dimensional simplicial complex pseudomanifolds with $f$ $d$-faces and $w$ vertices. (A simplicial complex pseudomanifold is a pseudomanifold which is also a simplicial complex.) It would also be interesting to estimate the number of connected $d$-dimensional pseudomanifolds with $f$ faces and $w$ vertices.

\subsection{Background on pseudomanifolds} \label{subsecbackpseudo}

In this section, we provide background on pseudomanifolds. We recall the definition of a pseudomanifold and a pseudomanifold with boundary. We will see that the spaces $X_k(\vec g)$ are all pseudomanifolds with boundary. We will define vertices, edges, faces, and connected components of a pseudomanifold. We will see that the boundary of a $d$-dimensional pseudomanifold with boundary is itself a $(d-1)$-dimensional pseudomanifold (without boundary). As a result, we will see that the boundary of $X_k(\vec g)$ consists of finitely many components and that each components is a loop with at least one edge. Finally, if $X(\vec g)$ is a triangulated surface, then we will check that each component of the boundary of $X_k(\vec g)$ contains at least two edges.

A $d$-dimensional pseudomanifold is made by gluing together $d$-dimensional simplices. Suppose that $\Delta_1, \Delta_2, ..., \Delta_f $ are copies of the standard (closed) $d$-simplex. (In this paper, we always work with finite pseudomanifolds, made from finitely many simplices.)

We glue facets of these $\Delta_j$ together using simplicial isomorphisms. Each $\Delta_j$ has $d+1$ facets, which we label as $\Delta_{j,a}$ with $a = 0, ..., d$. A gluing is defined by specifying two (different) facets,
$\Delta_{j,a}$ and $\Delta_{j', a'}$ and giving a simplicial isomorphism from $\Delta_{j,a}$ to $\Delta_{j', a'}$. (Technical remark. Gluing a facet $\Delta_{j,a}$ to itself is not allowed. But gluing a facet $\Delta_{j,a}$ to another facet of the same simplex, $\Delta_{j,a'}$, is allowed.)

A pseudomanifold is specified by a set of gluings where each facet of the $\Delta_j$ is involved in exactly one gluing. Recalling the definition of a gluing story, it follows immediately that for any gluing story $\vec g$, $X(\vec g)$ is a 2-dimensional pseudomanifold.

A pseudomanifold with boundary is specified by a set of gluings where each facet of the $\Delta_j$ is involved in at most one gluing. It follows that for each $\vec g \in GS(f)$ and each $k$, $X_k(\vec g)$ is a pseudomanifold with boundary.

A pseudomanifold (possibly with boundary) leads to an underlying topological space by identifying any points that have been glued together. The set of points in the pseudomanifold is formally defined as follows. We begin with the union of the simplices $\Delta_1, \Delta_2, ...$ If a gluing map takes one point to another point, then those points are equivalent. These equivalences generate equivalence classes. A point of the pseudomanifold is an equivalence class of the original points.

We can define the $k$-dimensional faces of a pseudomanifold (possibly with boundary) in a similar way. We begin with the set of $k$-dimensional faces of the simplices $\Delta_j$. Two $k$-dimensional faces are equivalent if one of our gluing maps maps one of them onto the other. These equivalences generate an equivalence relation on the set of $k$-faces. A $k$-face of the pseudomanifold is an equivalence class for this relation. In particular, we can define the vertices of a pseudomanifold as the 0-dimensional faces. Two $d$-dimensional faces can never be glued together, so the $d$-faces of a $d$-dimensional pseudomanifold are just the original $d$-dimensional simplices $\Delta_1, \Delta_2,$...

For example, a 1-dimensional pseudomanifold without boundary is a finite collection of circles, where each circle is made from some number of intervals connected end-to-end. A 1-dimensional pseudomanifold can be made from a single interval with its two boundary points glued together. So each circle can have any number of edges $\ge 1$.

Next we define connected pseudomanifolds. Two simplices $\Delta_j$ and $\Delta_{j'}$ are adjacent if two of their facets have been glued together. We say $\Delta_{j}$ is connected to $\Delta_{j'}$ if there is a sequence of adjacent simplices starting with $\Delta_j$ and ending with $\Delta_{j'}$. We say that a pseudomanifold is connected if every two simplices are connected. Any pseudomanifold (possibly with boundary) is a finite union of disjoint connected pseudomanifolds, its connected components.

Now we define the boundary of a pseudomanifold with boundary. Consider a $d$-dimensional pseudomanifold with boundary, $X$, formed from $d$-simplices $\Delta_1, \Delta_2, ...$ by some gluing maps $g_1, g_2, ...$ A facet $\Delta_{j,a}^{d-1} \subset \Delta^d_j$ is called a boundary simplex if it is not involved in any of the gluing maps. (We write the exponent $d-1$ to recall that the dimension of $\Delta_{j,a}^{d-1}$ is $d-1$.) It turns out that the boundary simplices form a $(d-1)$-dimensional pseudomanifold without boundary.

If $\Delta_{j,a}^{d-1}$ is a boundary simplex, and $\Delta_{j,a,b}^{d-2} \subset \Delta_{j,a}^{d-1}$ is one of its facets, then we have to prove that the equivalence class of $\Delta_{j,a,b}$ lies in exactly one other boundary simplex. We will describe the equivalence class of $\Delta_{j,a,b}$ in terms of the gluing maps.

The facets of a boundary simplex are $(d-2)$-dimensional. To help study them, we define some notation to describe the $(d-2)$-dimensional facets of a $d$-simplex. Recall that if $\Delta_j$ is one of our $d$-simplices, then its facets are $\Delta_{j,a}$ with $a = 0, ..., d$. Any $(d-2)$-face of $\Delta_j$ is contained in exactly two of the facets of $\Delta_j$. For any $a \not= b$, we write $\Delta_{j,a,b}^{d-2} := \Delta_{j,a}^{d-1} \cap \Delta_{j,b}^{d-1} \subset \Delta_j^d$.

We are considering a boundary simplex $\Delta_{j,a}$ and one of its facets $\Delta_{j,a,b}$.
If $\Delta_{j,b}^{d-1}$ is also a boundary simplex, then the equivalence class of $\Delta_{j,a,b}$ is just $\{ \Delta_{j,a,b} \}$, and it is a facet of $\Delta_{j,a}$ and $\Delta_{j,b}$. In this case, we are done.

Now suppose that $\Delta_{j,b}^{d-1}$ is not a boundary simplex, i.e. it is involved in a gluing. Suppose that $\Delta_{j,b}^{d-1}$ is glued to $\Delta_{j_1, a_1}^{d-1} \subset \Delta^d_{j_1}$. This gluing identifies $\Delta_{j,a,b}^{d-2}$ with a $(d-2)$-face $\Delta_{j_1, a_1, b_1}^{d-2} \subset \Delta_{j_1, a_1}^{d-1}$. As above, we use the convention that $\Delta_{j_1, a_1, b_1}^{d-2} = \Delta_{j_1, a_1}^{d-1} \cap \Delta_{j_1, b_1}^{d-1}$. If $\Delta_{j_1, b_1}$ is a boundary simplex, then the equivalence class of $\Delta_{j,a,b}$ is $\{ \Delta_{j,a,b}, \Delta_{j_1, a_1, b_1} \}$, and it is a facet of exactly two boundary simplices: $\Delta_{j,a}$ and $\Delta_{j_1, b_1}$.

If $\Delta_{j_1, b_1}$ is not a boundary simplex, then it is glued to some $\Delta_{j_2, a_2}$. This gluing map identifies $\Delta_{j,a,b}$ with some $(d-2)$-face $\Delta_{j_2, a_2, b_2}^{d-2} = \Delta_{j_2, a_2}^{d-1} \cap \Delta_{j_2, b_2}^{d-1} \subset \Delta_{j_2}$. The general picture is as follows. For $i = 1, ..., t$, $\Delta_{j_i, b_i}$ is glued to $\Delta_{j_{i+1}, a_{i+1}}$, which identifies $\Delta_{j_i, a_i, b_i}$ with $\Delta_{j_{i+1}, a_{i+1}, b_{i+1}}$. The face $\Delta_{j_t, b_t}$ is another boundary simplex. This procedure has to stop with another boundary simplex, $\Delta_{j_t, b_t}$, because the $(d-1)$-faces $\Delta_{j_i, a_i}$ and $\Delta_{j_i, b_i}$ are all distinct. Now the equivalence class of $\Delta_{j,a,b}$ is exactly $\{ \Delta_{j,a,b}, \Delta_{j_1, a_1, b_1}, ..., \Delta_{j_t, a_t, b_t} \}$. It is the facet of exactly two boundary simplices: $\Delta_{j,a}$ and $\Delta_{j_t, b_t}$.

This finishes our explanation of the structure of the boundary of a pseudomanifold with boundary. In particular, we see that $\partial X_k(\vec g)$ is a 1-dimensional pseudomanifold without boundary. By the classification we described above, $\partial X_k(\vec g)$ is a finite union of circles each containing at least one edge.

Finally, we check that if $X(\vec g)$ is a triangulated surface, then each component of $\partial X_k(\vec g)$ contains at least two edges. Since $X(\vec g)$ is a triangulated surface, each edge of $X(\vec g)$ has two distinct vertices. This implies that each edge of $X_k(\vec g)$ has two distinct vertices. Therefore, each edge of $\partial X_k(\vec g)$ has two distinct vertices, and so each component of $\partial X_k(\vec g)$ contains at least two edges.

\section{Comments} \label{sec:comm}

\subsection{Filling area of cycles in random $2$-complexes}

Over the past ten years or so, the topology of random $2$-complexes has been well studied. This area began with Linial and Meshulam's paper \cite{LM}. The Linial--Meshulam theorem describes the vanishing threshold for homology.

\begin{thm*}[Linial--Meshulam, 2006]
Let $Y = Y(n,p)$.\\

If $$p \ge \frac{ 2 \log n + \omega(1)}{n}$$
then
with high probability $H_1(Y, \Z / 2 \Z) = 0$, and
if $$p \le \frac{2 \log n - \omega(1)}{n},$$
then with high probability $H_1(Y, \Z / 2 \Z) \neq 0$. (Here $\omega(1)$ denotes any function that tends to infinity as $n \to \infty$.)
\end{thm*}

A geometric refinement of this picture would be to understand the typical {\it filling area} of cycles, i.e.\ in the typical number of triangles in the minimal bounding chain for a given cycle.  For instance, what is the typical filling area of the length 3 1-cycle $123$ in $Y$?  The Linial-Meshulam theorem gives an upper bound as follows.  Consider the sub complex of $Y$ restricted to the first $s$ vertices.  This sub-complex is chosen according to $Y(s,p)$.  If $p > s^{\eps - 1} > \frac{2 \log s}{s}$, then with high probability the first homology of the sub complex vanishes by the theorem of Linial-Meshulam.  Therefore, the cycle $123$ bounds a 2-chain in this sub complex.  But the number of 2-faces in the sub-complex is at most $C p s^3$ with high probability.  Therefore, with high probability, the filling area of $123$ is at most $C p^{-2 - \eps}$.  If $p = n^{\alpha - 1}$, then with high probability $Y$ has about $n^{2 + \alpha}$ 2-faces and the filling area of $123$ is at most $n^{2 - 2 \alpha + \eps}$.  The techniques of this paper allow one to prove that this upper bound is essentially sharp when $0 \le \alpha < 1/2$.

\begin{thm} \label{thm:fill}
Let $\alpha \in \left[ 0 , \frac{1}{2} \right)$ and $\epsilon > 0$. Suppose $p = n^{\alpha-1}$, so the expected number of faces in $Y(n,p)$ is $n^{2 + \alpha}$. Let $A$ be the filling area of the cycle $123$ in $Y$. Then with high probability
$$n^{2 - 2 \alpha - \epsilon} \le A \le n^{2 - 2 \alpha + \epsilon}.$$
\end{thm}

The new part of Theorem \ref{thm:fill} is the lower bound.  The proof of the lower bound is very similar to the proof of Theorem \ref{thm:exist}.  One counts fillings of the 1-cycle $123$ instead of counting 2-cycles, but the counts are closely related since adding a triangle to a filling results in a cycle.  The only difference has to do with small fillings or small cycles.  A small filling of $123$ gives rise to a small inclusion-minimal cycle which contains the vertices $123$.  The number of such cycles in $Y$ goes to zero, although the total number of small inclusion-minimal cycles does not.  Here is a more detailed explanation.

With high probability, $Y$ does not contain the face with vertices 123.  If $c$ is a 2-chain in $Y$ bounded by the cycle $123$ and with a minimal number of 2-faces, then adding the face with vertices 123 to $c$ gives an inclusion-minimal 2-cycle in $\Delta_n^{(2)}$.  

Recall that an inclusion-minimal 2-cycle was called small if it contains at most $M$ vertices for $M= M(\alpha)$ defined in Section \ref{sec:largegirth}.  Similarly, an inclusion minimal filling of $123$ is called small if it contains at most $M$ vertices.  The small cycles in $Y$ were estimated in Section \ref{sec:smallcycles} using barely-dense subcomplexes.  Recall that a barely dense subcomplex of $Y$ is a subcomplex with $v$ vertices and $2v-4$ faces.  Each small inclusion-minimal 2-cycle contains a barely dense subcomplex.  Similarly, an inclusion-minimal small filling of the cycle $123$ must contain a sub complex of $Y$ with $v$ vertices and $2v-5$ faces of $Y$, and containing the vertices $1, 2,$ and $3$.  The expected number of such sub-complexes in $Y$ is 

$$\sum_{v = 4}^{M} {n \choose {v-3}} {{v \choose 3} \choose 2v-4} p^{2v-5}.$$

This sum is decreasing in $v$, and so it is dominated by $M$ times the first term.  Hence the sum is

$$ \le C_{\alpha} n p^3 = C_{\alpha} n^{3 \alpha - 2}. $$

Since $\alpha < 1/2$, this is at most $C_{\alpha} n^{-1/2}$, and so with high probability, there is no small filling of $Y$.  The rest of the proof of Theorem \ref{thm:fill} is the same as the proof of Theorem \ref{thm:exist}.

\subsection{Comparison with earlier work}

Aronshtam, Linial, \L{}uczak, and Meshulam studied the threshold for collapsibility of random $d$-dimensional simplicial complexes in \cite{ALLM13}. In Section 4 they provide an upper bound on $C_d(n,m)$, the number of minimal core $d$-complexes with $n$ vertices and $m$ facets. This is in turn gives an upper bound on the number of inclusion-minimal cycles with $n$ vertices and $m$ facets, since every cycle is a minimal core.

So their estimate can be used to count cycles in random complexes. In the case we are interested in, $d=2$, their estimate implies the following.

\begin{thm} \label{thm:ALLM}
\noindent
\begin{enumerate}
\item For any $C > 0$, for all $n$ sufficiently large, there exist simplicial complexes $\Delta$ with $n$ vertices and with at least $Cn^2$ faces, such that every cycle in $H_2(\Delta)$ is supported on at least $f(C)n^2$ faces.\\

\item For $\alpha > 0$, for all $n$ sufficiently large, there exist simplicial complexes $\Delta$ with $n$ vertices and with at least $n^{2 + \alpha}$ faces, such that every cycle in $H_2(\Delta)$ is supported on at least $C_{\alpha} n^{2 - 16 \alpha}$ faces.
\end{enumerate}
\end{thm}

Part (1) of Theorem \ref{thm:ALLM} is slightly stronger than our Theorem \ref{thm:exist} in the regime $m = \theta \left( n^2 \right)$, and is optimal up to a constant factor by Theorem \ref{thm:syst}. On the other hand, part (2) of this theorem is weaker than Theorem \ref{thm:exist} for $\alpha > 0$, since $$2- 16 \alpha < 2 - 2 \alpha - \epsilon$$ for sufficiently small $\epsilon$.

\subsection{Volume distortion}

By combining Theorem \ref{thm:fill} with a standard probabilistic technique (see for example the proof of Proposition 4.2 in \cite{Dominic-distortion}), another estimate is obtained:

\begin{prop}\label{Weiner-index}
Let $Y = Y(n,p)$ be a random complex with $p = n^{\alpha-1}$. Let $\epsilon >0 $ be given. For each $1$-cycle, $\tau$, of length $3$ (e.g. $123$ as above), let $${\rm Fill}_Y (\tau) = \min \{ \Vert y \Vert \, | \, y \in C_2 (Y; \Z / 2 \Z) \text{  and   } \partial y = \tau \}.$$ Then for sufficiently large $n$ (depending on $\epsilon$) with high probability $$\sum_{\tau \in \binom{n}{3} } \left( {\rm Fill}_Y (\tau) \right)^2 \geq \binom{n}{3} \left( n^{2-2\alpha -\epsilon}\right)^2.$$
\end{prop}

This leads immediately to an improvement of Theorem 1.2 in \cite{Dominic-distortion}:

\begin{thm}{\label{opt-dist}}  Let $Y(n,p)$ be as in Proposition \ref{Weiner-index}.  Let $\epsilon > 0$.  Then with high probability, for every map $\phi: Y \rightarrow \mathcal{H}$ from $Y$ to a Hilbert space $\mathcal{H}$ which is affine on each simplex,

$$ \max_{\tau \in \binom{n}{3}} \frac{ {\rm Fill}_Y (\tau)}{{\rm Fill}_\mathcal{H} (\phi \tau) } \cdot \max_{\tau \in \binom{n}{3}} \frac{{\rm Fill}_\mathcal{H} (\phi \tau) }{{\rm Fill}_Y (\tau) } \geq c n^{2- 2 \alpha - \epsilon} , $$

where ${\rm Fill}_\mathcal{H} (\phi \tau)$ refers to the area of the convex hull of $\phi (\tau)$.  In particular, by taking $\alpha > 0$ arbitrarily small, we see that for every $\delta > 0$, and for every large enough $n > C = C_\delta$, there exists a $2$-dimensional simplicial complex, $Y$ with complete $1$-skeleton, such that every map, $\phi : Y \to \mathcal{H}$, from $Y$ to a Hilbert space $\mathcal{H}$, which is affine on each simplex has that

$$ \max_{\tau \in \binom{n}{3}} \frac{ {\rm Fill}_Y (\tau)}{{\rm Fill}_\mathcal{H} (\phi \tau) } \cdot \max_{\tau \in \binom{n}{3}} \frac{{\rm Fill}_\mathcal{H} (\phi \tau) }{{\rm Fill}_Y (\tau) } \geq cn^{2-\delta}.$$

\end{thm}

The left hand side of this inequality is referred to in \cite{Dominic-distortion} as the {\it filling distortion} of $\phi$ and is a natural homological analogue of much studied phenomenon of metric distortion (see \cite{LMN2002girth} for a survey).

The following is one way to think about this theorem. For any complex $Y$ with complete $1$-skeleton, we could embed that complex into Euclidean space by sending the $n$ vertices to the standard basis elements of $\R^n$. In that case, the convex hull of any triangle has area $O(1)$ and filling distortion is bounded from below by $\max_{\tau} {\rm Fill}_Y (\tau)$, which is at most $\vert Y^{(2)} \vert$.  If we choose $\alpha > 0$ small and positive, then with high probability $| Y^{(2)} | \sim n^{2 + 2 \alpha}$ and the distortion is at least around $n^{2 - 2 \alpha}$, which are very close to each other.  For such $Y$, this naive, ``equilateral'' embedding achieves essentially the best result.

The analogous statement in the context of graphs is the following: for expander graphs with $n$ vertices, every embedding into Euclidean requires $\Omega (\log n)$ metric distortion, and since the diameter of such an expander graph is at most $O(\log n)$, an embedding that sends the vertices to the standard basis of $\R^n$ achieves essentially the minimal possible distortion.

\bibliography{sysrefs}

\end{document}